\documentclass[a4paper,12pt]{amsart}

\usepackage[utf8]{inputenc} 
\usepackage{amsmath}
\usepackage{amsbsy}
\usepackage{amssymb}
\usepackage{amscd,amsthm}
\usepackage{verbatim}
\usepackage[english]{babel}
\usepackage{dsfont}
\usepackage{anysize}
\usepackage{hyperref}

\newcommand{\e}{\varepsilon}

\newcommand{\C}{\mathds{C}}

\newtheorem{thm}{Theorem}[section]
\newtheorem{lem}[thm]{Lemma}

\theoremstyle{definition}

\theoremstyle{remark}

\title{On Ramanujan's prime counting inequality}
\author{Christian Axler}
\address{Institute of Mathematics\\ Heinrich Heine University Duesseldorf\\
40225 Duesseldorf, Germany}
\email{christian.axler@hhu.de}
\date{\today}
\subjclass[2010]{Primary 11N05; Secondary 11M26}
\keywords{Prime counting function, Ramanujan's prime counting inequality, Riemann hypothesis}

\begin{document}

\begin{abstract}
In this paper, we give a new upper bound for the number $N_{\mathcal{R}}$ which is 
defined to be the smallest positive integer such that a certain inequality due to Ramanujan involving the prime counting function $\pi(x)$ holds for every $x \geq N_{\mathcal{R}}$.
\end{abstract}

\maketitle


\section{Introduction}

Let $\pi(x)$ denote the number of primes not exceeding $x$. Since there are infinitely many primes, we have $\pi(x) \to \infty$ as $x \to \infty$. Gauss \cite{gauss} stated a conjecture concerning an asymptotic behavior of $\pi(x)$, namely
\begin{equation}
\pi(x) \sim \text{li}(x) \quad\quad (x \to \infty), \tag{1.1} \label{1.1}
\end{equation}
where the \emph{logarithmic integral} $\text{li}(x)$ is defined as
\begin{equation}
\text{li}(x) = \int_0^x \frac{\text{d}t}{\log t} = \lim_{\e \to 0} \left \{ \int_{0}^{1-\e}{\frac{\text{d}t}{\log t}} + \int_{1+\e}^{x}{\frac{\text{d}t}{\log 
t}} \right \}, \tag{1.2} \label{1.2}
\end{equation}
where $\log x$ is the natural logarithm of $x$. Hadamard \cite{hadamard1896} and de la Vall\'{e}e-Poussin \cite{vallee1896} independently proved the asymptotic formula \eqref{1.1} which is known as the \textit{Prime Number Theorem}. In a later paper, \cite{vallee1899}
where the existence of a zero-free region for the Riemann zeta-function $\zeta(s)$ to the left of the line $\text{Re}(s) = 1$ was proved, de la Vall\'{e}e-Poussin also estimated the error term in the Prime Number Theorem by showing
\begin{equation}
\pi(x) = \text{li}(x) + O(x \exp(-\delta_0\sqrt{\log x})) \tag{1.3} \label{1.3}
\end{equation} 
as $x \to \infty$, where $\delta_0$ is a positive absolute constant. Integration by parts in \eqref{1.3} implies that for every positive integer $n$, we have
\begin{equation}
\pi(x) = \frac{x}{\log x} + \frac{x}{\log^2 x} + \frac{2x}{\log^3 x} + \frac{6x}{\log^4 x} + \ldots + \frac{(n-1)! x}{\log^nx} + O \left( \frac{x}{\log^{n+1} x} \right) \tag{1.4} \label{1.4}
\end{equation}
as $x \to \infty$. In one of his notebooks (see Berndt \cite{berndt1994}), 
Ramanujan used \eqref{1.4} with $n = 5$ to find that
\begin{displaymath}
\pi(x)^2 - \frac{ex}{\log x} \pi \left( \frac{x}{e} \right)  = - \frac{x^2}{\log^6x} + O \left( \frac{x}{\log^7 x} \right) \tag{1.5} \label{1.5}
\end{displaymath}
as $x \to \infty$ and concluded that the inequality
\begin{equation}
\pi(x)^2 < \frac{ex}{\log x} \pi \left( \frac{x}{e} \right) \tag{1.6} \label{1.6}
\end{equation}
holds for all sufficiently large values of $x$. The inequality \eqref{1.6} is called \textit {Ramanujan's prime counting inequality}.
Recently, Hassani \cite[Corollary 1]{hassani2} found the full asymptotic expansion in \eqref{1.5} by showing that for every integer $n \geq 4$, one has
\begin{displaymath}
\pi(x)^2 - \frac{ex}{\log x} \pi \left( \frac{x}{e} \right) = x^2 \sum_{k=4}^n \frac{d_k}{\log^{k+2}x} + O \left(\frac{x^2}{\log^{n+3} x} \right)
\end{displaymath}
as $x \to \infty$, where
\begin{displaymath}
d_k = \sum_{j=0}^k j! \left( (k-j)! - \binom{k}{j} \right).
\end{displaymath}
A legitimate question is to find the smallest positive integer $N_{\mathcal{R}}$ so that the inequality \eqref{1.6} holds for every real $x \geq N_{\mathcal{R}}$.
The first result made in the search for $N_{\mathcal{R}}$ is based on the assumption that the Riemann hypothesis is true. The Riemann zeta function is for all complex numbers $s$ with $\text{Re}(s) > 1$ defined as
\begin{displaymath}
\zeta(s) = \sum_{n=1}^s \frac{1}{n^s}.
\end{displaymath}
It is well known that the Riemann zeta function is a meromorphic function on the whole complex plane, which is holomorphic everywhere except for a simple pole at $s = 1$.
The Riemann zeta function satisfies the functional equation
\begin{displaymath}
\zeta(s)= 2^s \pi^{s-1} \sin \left({\frac{\pi s}{2}} \right) \Gamma(1-s)\zeta (1-s),
\end{displaymath}
where $\Gamma(s)$ is the gamma function. This is an equality of meromorphic functions valid on the whole complex plane. Due to the zeros of the sine function, the functional equation implies that $\zeta(s)$ has outside the set $\{ s \in \C \mid 0 \leq \text{Re}(s) \leq 1 \}$ a simple zero at each even negative integer $s = -2n$, known as the \textit{trivial zeros}. The \textit{nontrivial zeros}, i.e. the zeros in the set $\{ s \in \C \mid 0 \leq \text{Re}(s) \leq 1 \}$, have attracted far more attention because not only is their distribution far less well known, but their study also yields important results concerning primes and related objects in number theory.
The \textit{Riemann hypothesis} asserts that the real part of every nontrivial zero of the Riemann zeta function is $1/2$.
To this day, the Riemann hypothesis is considered one of the greatest unsolved problems in mathematics.
Under the assumption that the Riemann hypothesis is 
true (RH), Hassani \cite[Theorem 1.2]{hassani2012} has given the upper bound 
\begin{displaymath}
RH \; \Rightarrow \; N_{\mathcal{R}} \leq 138\,766\,146\,692\,471\,228.
\end{displaymath}
Dudek and Platt \cite[Lemma 3.2]{dudekplatt} refined Hassani's result by showing that
\begin{equation}
RH \; \Rightarrow \; N_{\mathcal{R}} \leq 1.15 \cdot 10^{16}. \tag{1.7} \label{1.7}
\end{equation}
Wheeler, Keiper and Galway (see Berndt \cite[p.\:113]{berndt1994}) attempted to determine the value of $N_{\mathcal{R}}$, but they failed. Nevertheless, Galway found that 
the largest prime up to $10^{11}$ for which the inequality \eqref{1.6} fails is $x=38\,358\,837\,677$. Hence
\begin{displaymath}
N_{\mathcal{R}} > 38\,358\,837\,677.
\end{displaymath}
Dudek and Platt \cite[Theorem 1.3]{dudekplatt} showed by computation that the largest (not necessarily prime) integer counterexample of \eqref{1.6} up to $x = 10^{11}$ occurs at $x = 38\,358\,837\,682$ and that the inequality \eqref{1.6} holds for every $x$ satisfying $10^{11} \leq x \leq 1.15 \times 10^{16}$. If we combine this result with \eqref{1.7}, it turns out that
\begin{displaymath}
RH \; \Rightarrow \; N_{\mathcal{R}} = 38\,358\,837\,683.
\end{displaymath}
Based on a result of Büthe \cite[Theorem 2]{buethe}, the present author \cite[Theorem 3]{axler} extends the computation of Dudek and Platt by showing that Ramanujan's prime counting inequality \eqref{1.6} holds unconditionally for every $x$ such that $38\,358\,837\,683 \leq x \leq 10^{19}$. This was improved by Platt and Trudgian \cite[Theorem 2]{plattt}. They showed that Ramanujan's prime counting inequality \eqref{1.6} holds unconditionally for every $x$ satisfying $38\,358\,837\,683 \leq x \leq e^{58}$. Recently, Johnston \cite[Theorem 5.1]{johnston} utilized a simple (but computationally intensive) method to show that the inequality \eqref{1.6} holds unconditionally for every $x$ with $38\,358\,837\,683 \leq x \leq e^{103}$.

In another direction, Dudek and Platt \cite[Theorem 1.2]{dudekplatt} claimed to give an upper bound for $N_{\mathcal{R}}$ which does not depend on the assumption that the Riemann hypothesis is true, namely $N_{\mathcal{R}} \leq \exp(9658)$.
After the present author identified an error in the proof given in \cite[Theorem 1.2]{dudekplatt}, he proved \cite[Theorem 4]{axler} the even stronger result
\begin{equation}
N_{\mathcal{R}} \leq \exp(9032). \tag{1.8} \label{1.8}
\end{equation}
In the proof of \eqref{1.8}, effective estimates for the prime counting function $\pi(x)$ which hold for all sufficiently large values of $x$ play an important role. Using their effective estimates for $|\vartheta(x)-x|$, where Chebyshev's $\vartheta$-function which is defined by
\begin{equation}
\vartheta(x) = \sum_{p \leq x} \log p, \tag{1.9} \label{1.9}
\end{equation}
where $p$ runs over all primes not exceeding $x$, Platt and Trudgian \cite[Theorem 2]{plattt} fixed the error in the proof given in \cite[Theorem 1.2]{dudekplatt} to show that Ramanujan's prime counting inequality \eqref{1.6} holds unconditionally for every $x \geq \exp(3915)$; i.e. 
\begin{equation}
N_{\mathcal{R}} \leq \exp(3915).  \tag{1.10} \label{1.10}
\end{equation}
Cully-Hugill and Johnston \cite[Corollary 1.6]{cully} used the method investigated by Platt and Trudgian to prove that
\begin{equation}
N_{\mathcal{R}} \leq \exp(3604).  \tag{1.11} \label{1.11}
\end{equation}
Recently, Johnston and Yang \cite[Theorem 1.5]{yang} utilized the same method to improve the last result by showing that
\begin{equation}
N_{\mathcal{R}} \leq \exp(3361).  \tag{1.12} \label{1.12}
\end{equation}
In this paper we will also make use of this method combined with a recent result concerning the difference of $\vartheta(x)$ and $x$ due to Fiori, Kadiri, and Swidinsky \cite{kadiri} to show the following

\begin{thm} \label{thm101}
Ramanujan's prime counting inequality \eqref{1.6} holds unconditionally for every $x \geq \exp(3158.442)$; i.e.
\begin{displaymath}
N_{\mathcal{R}} \leq \exp(3158.442).
\end{displaymath}
\end{thm}

\section{Preliminaries}

The prime counting function $\pi(x)$ and Chebyshev's $\vartheta$-function (cf. \eqref{1.9}) are connected by the 
identity
\begin{equation}
\pi(x) = \frac{\vartheta(x)}{\log x} + \int_{2}^{x}{\frac{\vartheta(t)}{t \log^{2} t}\ \text{d}t}, \tag{2.1} \label{2.1}
\end{equation}
which holds for every $x \geq 2$ (see \cite[Theorem 4.3]{ap}). The method established to prove results like \eqref{1.10}-\eqref{1.12} or Theorem \ref{thm101} goes back to Dudek and Platt \cite{dudekplatt} and is as follows. We start with a function $a$ so that
\begin{equation}
|\vartheta(x)-x| \leq \frac{a(x)x}{\log^5x} \tag{2.2} \label{2.2}
\end{equation}
for every $x \geq x_0$. Then we substitute this inequality into \eqref{2.1} to derive upper and lower bounds for the prime counting function $\pi(x)$ of the form
\begin{equation}
x\sum_{k=0}^4 \frac{k!}{\log^{k+1}x} + \frac{m_a(x)x}{\log^6x} \leq \pi(x) \leq x\sum_{k=0}^4 \frac{k!}{\log^{k+1}x} + \frac{M_a(x)x}{\log^6x} \tag{2.3} \label{2.3}
\end{equation}
for every $x \geq x_1$, where the functions $m_a(x)$ and $M_a(x)$. Using these estimates, Dudek and Platt were able to give an
explicit version of \eqref{1.5}. As already mentioned in the introduction, we also use this method. So, we need to find a function $a$ and a positive real number $x_0$ so that the inequality \eqref{2.2} holds. For this purpose, we set
\begin{displaymath}
R = 5.5666305
\end{displaymath}
and, similar to \cite[p.\:879]{plattt}, we define the function $a: \mathbb{R}_{>0} \to \mathbb{R}$ by
\begin{displaymath}
\frac{a(x)}{\log^5x} = 
\begin{cases}
\displaystyle \frac{2 - \log2}{2} & \text{if $2 \leq x < 599$,} \\
\displaystyle\frac{\log^2x}{8\pi\sqrt{x}} & \text{if $599 \leq x < 1.101 \times 10^{26}$,} \\
\displaystyle \sqrt{\frac{8}{17\pi}} \left( \frac{\log x}{6.455} \right)^{1/4} \exp \left( - \sqrt{\frac{\log x}{6.455}} \right) & \text{if $1.101 \times 10^{26} \leq x < e^{673}$,} \\
\displaystyle 121.0961 \left( \frac{\log x}{R} \right)^{3/2} \exp \left( - 2 \sqrt{\frac{\log x}{R}} \right) & \text{if $x \geq e^{673}$.}
\end{cases}
\end{displaymath}
Then we get the following result concerning Chebyshev's $\vartheta$-function.

\begin{lem} \label{lem201}
For every $x \geq 2$, we have
\begin{displaymath}
|\vartheta(x) - x| \leq \frac{a(x)x}{\log^5x}.
\end{displaymath}
\end{lem}

\begin{proof}
If $x$ satisfies $2 \leq x < 599$, then the given bound is trivial. The second one was proven by Johnston \cite[Corollary 3.3]{johnston} and the third bound was given by Trudgian \cite[Theorem 1]{trud}. The last bound was recently established by Fiori, Kadiri, and Swidinsky \cite[Corollary 14]{kadiri}.
\end{proof}

We also need the following result on our function $a$.


\begin{lem} \label{lem202}
Let $x_1$ be real number with $x_1 \geq e^{673}$. Then $a_n(x) \leq a_n(x_1)$ for every $x \geq x_1$.
\end{lem}

\begin{proof}
By a straightforward calculation of the derivative, we see that the inequality $a'(x) < 0$ holds for every $x \geq e^{673}$.
\end{proof}

Now, let $x_1$ be a real number with $x_1 \geq e^{673}$. According to the method we use, we set
\begin{align*}
C_0 &= \int_2^{x_1} \frac{720 - a(t)}{\log^7 t} \, \text{d}t - 2\sum_{k=1}^5 \frac{k!}{\log^{k+1}2}, \\
C_1 &= \int_2^{x_1} \frac{720 + a(t)}{\log^7 t} \, \text{d}t - 2\sum_{k=1}^5 \frac{k!}{\log^{k+1}2}. 
\end{align*}
Contrary to Dudek and Platt \cite{dudekplatt}, who have estimated the integral
\begin{displaymath}
\int_2^x \frac{\text{d}t}{\log^7t},
\end{displaymath}
we will use the identity
\begin{equation}
\int_{x_1}^x \frac{\text{d}t}{\log^7t} = E(x) - E(x_1), \tag{2.4} \label{2.4}
\end{equation}
where
\begin{displaymath}
E(x) = \frac{1}{720} \left( \text{li}(x) - \frac{x}{\log x} - \frac{x}{\log^2x} - \frac{2x}{\log^3x} - \frac{6x}{\log^4x} - \frac{24x}{\log^5x} - \frac{120x}{\log^6x}\right),
\end{displaymath}
to find the following explicit version of \eqref{2.3}.

\begin{lem} \label{lem203}
Let
\begin{displaymath}
m_a(x) = 120 - a(x) + \left( C_0 + (720 - a(x_1))(E(x) - E(x_1)) \right) \frac{\log^6x}{x}
\end{displaymath}
and 
\begin{displaymath}
M_a(x) = 120 + a(x) + \left( C_1 + (720 + a(x_1))(E(x) - E(x_1)) \right) \frac{\log^6x}{x}.
\end{displaymath}
Then
\begin{displaymath}
x\sum_{k=0}^4 \frac{k!}{\log^{k+1}x} + \frac{m_a(x)x}{\log^6x} \leq \pi(x) \leq x\sum_{k=0}^4 \frac{k!}{\log^{k+1}x} + \frac{M_a(x)x}{\log^6x}
\end{displaymath}
for every $x \geq x_1$.
\end{lem}

\begin{proof}
We only give a proof of the required upper bound. The proof of the required lower bound is similar and we leave the details to the reader. Let $x \geq x_1$. If we combine \eqref{2.1} with Lemma \ref{lem201}, we can see that
\begin{displaymath}
\pi(x) \leq \frac{x}{\log x} + \frac{xa(x)}{\log^6x} + \int_2^x \frac{\text{d}t}{\log^2 t} + \int_2^x \frac{a(t)}{\log^7 t} \, \text{d}t.
\end{displaymath}
Integration by parts in \eqref{1.3} provides that
\begin{displaymath}
\pi(x) \leq x\sum_{k=0}^4 \frac{k!}{\log^{k+1}x} + \frac{(120+a(x))x}{\log^6x} + C_1 + \int_{x_1}^x \frac{720 + a(t)}{\log^7 t} \, \text{d}t.
\end{displaymath}
Since $a(t) \leq a(x_1)$ for every $t$ with $x_1 \leq t \leq x$ (cf. Lemma \ref{lem202}), it turns out that
\begin{displaymath}
\pi(x) \leq x\sum_{k=0}^4 \frac{k!}{\log^{k+1}x} + \frac{(120+a(x))x}{\log^6x} + C_1 + (720 + a(x_1))\int_{x_1}^x \frac{\text{d}t}{\log^7 t}.
\end{displaymath}
Finally, it suffices to apply the identity \eqref{2.4}.
\end{proof}

\section{Proof of Theorem \ref{thm101}}

Now we have all the necessary tools to give a proof of Theorem \ref{thm101}.

\begin{proof}[Proof of Theorem \ref{thm101}]
Let $x_1 = \exp(3157.442)$. Then, one has $a(x_1) = 1056.767676\ldots$ and $E(x_1) > 0$. Since the function $E(x)\log^6(x)/x$ is decreasing on the interval $[x_1, \infty)$ and $a(x) \leq a(x_1)$ for every $x \geq x_1$ (cf. Lemma \ref{lem202}), we can use Lemma \ref{lem203} and a computer to get that
\begin{displaymath}
x\sum_{k=0}^4 \frac{k!}{\log^{k+1}x} + \frac{m_a x}{\log^6x} \leq \pi(x) \leq x\sum_{k=0}^4 \frac{k!}{\log^{k+1}x} + \frac{M_a x}{\log^6x}
\end{displaymath}
for every $x \geq x_1$, where
\begin{align*}
m_a &= -936.64603213534,\\
M_a &= 1177.56019022252.
\end{align*}
Now we can argue as in the proof of \cite[Lemma 2.1]{dudekplatt} to see that
\begin{equation}
\pi^2(x) - \frac{ex}{\log x}\, \pi \left( \frac{x}{e} \right) < \frac{x^2}{\log^6x} \left( -1 + \frac{\epsilon_{M_a}(x) - \epsilon_{m_a}(x)}{\log x} \right) \tag{3.1} \label{3.1}
\end{equation}
for every $x \geq ex_1$, where
\begin{align*}
\epsilon_{m_a}(x) &= 206 + m_a + \frac{364}{\log x} + \frac{381}{\log^2x} + \frac{238}{\log^3x} + \frac{97}{\log^4x} + \frac{30}{\log^5x} + \frac{8}{\log^6x}, \\
\epsilon_{M_a}(x) &= 72 + 2M_a + \frac{2M_a + 132}{\log x} + \frac{4M_a + 288}{\log^2 x} + \frac{12M_a + 576}{\log^3 x} + + \frac{48M_a}{\log^4 x} + \frac{M_a^2}{\log^5x}.
\end{align*}
Note that $\epsilon_{M_a}(x) - \epsilon_{m_a}(x) < \log x$ for every $x \geq ex_1 = \exp(3158.442)$. Finally, it suffices to substitute the last inequality into \eqref{3.1} and we arrive at the end of the proof.
\end{proof}

\section{Future work}
It is natural to ask whether we can derive comparable results if we replace the number $e$ in \eqref{1.6} by an arbitrary positive real number $\alpha$. In this context, Hassani \cite[Theorem 3]{hassani2} was able to show that if $\alpha > e$ then one has
\begin{displaymath}
\pi(x)^2 < \frac{\alpha x}{\log x} \pi \left( \frac{x}{\alpha} \right) \tag{4.1} \label{4.1}
\end{displaymath}
for all sufficiently large values of $x$ and if $0 < \alpha < e$, then the above inequality reverses. One could inverstigated a method, similar to the one we used in the proof of Theorem \ref{thm101}, to find effective estimates for the smallest positive integer $N_{\mathcal{R}, \alpha}$ so that the inequality \eqref{4.1} holds for every real $x \geq N_{\mathcal{R}, \alpha}$.

\section*{Acknowledgement}

The author wishes to thank the two beautiful souls R. and O. for the never ending inspiration.

\end{document}